\documentclass[12pt]{amsart}

\usepackage{geometry} 
\usepackage[pdfpagemode=UseNone,pdfstartview=FitH,hypertex]{hyperref}

\newcommand{\A}{{\mathcal A}}
\newcommand{\Ap}[1][]{A_p #1}
\newcommand{\Bp}[1][]{B_p #1}
\newcommand{\ds}[1]{\displaystyle #1}
\newcommand{\dualp}[3][]{\left(#2,#3\right)_{#1}}
\newcommand{\eps}{\varepsilon}
\newcommand{\F}{{\mathcal F}}
\newcommand{\incl}{\hookrightarrow}
\newcommand{\N}{\mathbb N}
\newcommand{\norm}[2][]{\left\|#2\right\|_{#1}}
\renewcommand{\o}{\text{o}}
\newcommand{\PS}[1]{$(\text{PS})_{#1}$}
\newcommand{\pnorm}[2][]{\if #1'' \left|#2\right|_p \else \left|#2\right|_{#1} \fi}
\newcommand{\R}{\mathbb R}
\newcommand{\RP}{\R \text{P}}
\newcommand{\seq}[1]{\left(#1\right)}
\newcommand{\set}[1]{\left\{#1\right\}}
\newcommand{\wto}{\rightharpoonup}
\newcommand{\Z}{\mathbb Z}

\newenvironment{enumroman}{\begin{enumerate}

}{\end{enumerate}}

\newtheorem{corollary}{Corollary}[section]
\newtheorem{lemma}[corollary]{Lemma}
\newtheorem{proposition}[corollary]{Proposition}
\newtheorem{theorem}[corollary]{Theorem}

\theoremstyle{remark}
\newtheorem{example}[corollary]{Example}

\numberwithin{equation}{section}

\title[Nonlocal critical elliptic equations]{Nonlocal critical elliptic equations in homogeneous fractional Sobolev spaces}

\author[S. Carl]{Siegfried Carl}
\address[S. Carl]{Institut f\"ur Mathematik, Martin-Luther-Universit\"at Halle-Wittenberg,
D-06099 Halle, Germany}
\email{\tt siegfried.carl@mathematik.uni-halle.de}

\author[K. Perera]{Kanishka Perera}
\address[K. Perera]{Department of Mathematical Sciences, Florida Institute of Technology,
150 W University Blvd, Melbourne, FL 32901, USA}
\email{\\ kperera@fit.edu}

\author[H. Tehrani]{Hossein Tehrani}
\address[H. Tehrani]{Department of Mathematical Sciences, University of Nevada Las Vegas Box 454020, USA}
\email{\tt tehranih@unlv.nevada.edu}

\date{}


\begin{document}

\begin{abstract}
We prove new multiplicity results for some nonlocal critical growth elliptic equations in homogeneous fractional Sobolev spaces. The proofs are based on an abstract critical point theorem based on the $\Z_2$-cohomological index and on a novel regularity result for fractional $p$-Laplacian equations as well as on some compact embeddings.
\end{abstract}

\maketitle

\begin{center}
\begin{minipage}{12cm}
\tableofcontents
\end{minipage}
\end{center}

\section{Introduction}

In this paper we prove new multiplicity results for the critical fractional $p$-Laplacian equation
\begin{equation} \label{1.1}
(- \Delta_p)^s\, u = \lambda\, a(x)\, |u|^{p - 2}\, u + \mu\, b(x)\, |u|^{q - 2}\, u + |u|^{p_s^\ast - 2}\, u \quad \text{in } \R^N,
\end{equation}
where $(- \Delta_p)^s$ is the fractional $p$-Laplacian operator defined on smooth functions by
\[
(- \Delta_p)^s\, u(x) = 2 \lim_{\eps \searrow 0} \int_{\R^N \setminus B_\eps(x)} \frac{|u(x) - u(y)|^{p-2}\, (u(x) - u(y))}{|x - y|^{N+sp}}\, dy, \quad x \in \R^N,
\]
$s \in (0,1)$, $1 < p < N/s$, $p_s^\ast = Np/(N - sp)$ is the fractional critical Sobolev exponent, $p < q < p_s^\ast$, $a, b : \R^N \to \R$ are measurable functions that are not identically zero and satisfy
\begin{equation} \label{1.2}
0 \le a(x) \le \frac{c_a}{1 + |x|^{N + \alpha}}, \quad 0 \le b(x) \le \frac{c_b}{1 + |x|^{N + \beta}} \quad \text{for a.a.\! } x \in \R^N
\end{equation}
for some $\alpha, \beta > 0$ and $c_a, c_b \ge 0$, and $\lambda, \mu \ge 0$ are parameters.

Let
\[
[u]_{s,\,p} = \left(\int_{\R^{2N}} \frac{|u(x) - u(y)|^p}{|x - y|^{N+sp}}\, dx dy\right)^{1/p}
\]
be the Gagliardo seminorm of a measurable function $u : \R^N \to \R$ and let
\[
D^{s,\,p}(\R^N) = \set{u \in L^{p_s^\ast}(\R^N) : [u]_{s,\,p} < \infty}
\]
be the homogeneous fractional Sobolev space endowed with the norm $\norm{\cdot} = [\cdot]_{s,\,p}$. A weak solution of equation \eqref{1.1} is a function $u \in D^{s,\,p}(\R^N)$ satisfying
\begin{multline*}
\int_{\R^{2N}} \frac{|u(x) - u(y)|^{p - 2}\, (u(x) - u(y))\, (v(x) - v(y))}{|x - y|^{N + sp}}\, dx dy = \lambda \int_{\R^N} a(x)\, |u|^{p - 2}\, uv\, dx\\[5pt]
+ \mu \int_{\R^N} b(x)\, |u|^{q - 2}\, uv\, dx + \int_{\R^N} |u|^{p_s^\ast - 2}\, uv\, dx
\end{multline*}
for all $v \in D^{s,\,p}(\R^N)$. Weak solutions coincide with critical points of the $C^1$-functional
\begin{multline} \label{1.3}
E(u) = \frac{1}{p} \int_{\R^{2N}} \frac{|u(x) - u(y)|^p}{|x - y|^{N+sp}}\, dx dy - \frac{\lambda}{p} \int_{\R^N} a(x)\, |u|^p\, dx - \frac{\mu}{q} \int_{\R^N} b(x)\, |u|^q\, dx\\[5pt]
- \frac{1}{p_s^\ast} \int_{\R^N} |u|^{p_s^\ast}\, dx, \quad u \in D^{s,\,p}(\R^N).
\end{multline}

The eigenvalue problem
\begin{equation} \label{1.4}
(- \Delta_p)^s\, u = \lambda\, a(x)\, |u|^{p - 2}\, u \quad \text{in } \R^N
\end{equation}
will play a major role in our results. More specifically, our theorems will involve a certain increasing and unbounded sequence $\seq{\lambda_k}$ of eigenvalues of this problem (see Theorem \ref{Theorem 2.7}). First we will show that equation \eqref{1.1} has $m$ distinct pairs of nontrivial solutions for all $\lambda > 0$ in a suitably small left neighborhood of an eigenvalue of multiplicity $m \ge 1$. Let
\begin{equation} \label{1.5}
S = \inf_{u \in D^{s,\,p}(\R^N) \setminus \set{0}}\, \frac{[u]_{s,\,p}^p}{\pnorm[p_s^\ast]{u}^p}
\end{equation}
be the best fractional Sobolev constant and let
\[
\delta = \frac{S}{\pnorm[N/sp]{a}},
\]
where $\pnorm[r]{\cdot}$ denotes the norm in $L^r(\R^N)$ for $r \in [1,\infty]$.

\begin{theorem} \label{Theorem 1.1}
If $\lambda_k = \cdots = \lambda_{k+m-1} < \lambda_{k+m}$ for some $k, m \ge 1$, then for all $\lambda \in (\lambda_k - \delta,\lambda_k)$ and each $\mu \ge 0$, equation \eqref{1.1} has $m$ distinct pairs of nontrivial solutions $\pm u^\lambda_j,\, j = 1,\dots,m$ such that $E(u^\lambda_j) > 0$ and $u^\lambda_j \to 0$ in $D^{s,\,p}(\R^N) \cap L^\infty(\R^N)$ as $\lambda \nearrow \lambda_k$.
\end{theorem}

In particular, we have the following existence result when $m = 1$ (since $\lambda_k \nearrow \infty$, by taking $k$ larger if necessary, we may assume that $\lambda_k < \lambda_{k+1}$).

\begin{corollary}
For each $k \ge 1$, for all $\lambda \in (\lambda_k - \delta,\lambda_k)$ and each $\mu \ge 0$, equation \eqref{1.1} has a pair of nontrivial solutions $\pm u^\lambda$ such that $E(u^\lambda) > 0$ and $u^\lambda \to 0$ in $D^{s,\,p}(\R^N) \cap L^\infty(\R^N)$ as $\lambda \nearrow \lambda_k$.
\end{corollary}

We will also show that equation \eqref{1.1} has arbitrarily many solutions for all sufficiently large $\mu > 0$ under an additional assumption on the coefficients $a(x)$ and $b(x)$.

\begin{theorem} \label{Theorem 1.3}
Assume that
\[
\frac{a}{b^{p/q}} \in L^{q/(q-p)}(\R^N)
\]
and let $\lambda \ge 0$. Given any $m \ge 1$, $\exists \mu_{m,\,\lambda} > 0$ such that for all $\mu > \mu_{m,\,\lambda}$, equation \eqref{1.1} has $m$ distinct pairs of nontrivial solutions $\pm u^\mu_j,\, j = 1,\dots,m$ such that $E(u^\mu_j) > 0$ and $u^\mu_j \to 0$ in $D^{s,\,p}(\R^N)$ as $\mu \nearrow \infty$. In particular, the number of solutions goes to infinity as $\mu \nearrow \infty$.
\end{theorem}

In particular, we have the following multiplicity result for the equation
\begin{equation} \label{1.6}
(- \Delta_p)^s\, u = \frac{\lambda\, |u|^{p - 2}\, u}{1 + |x|^{N + \alpha}} + \frac{\mu\, |u|^{q - 2}\, u}{1 + |x|^{N + \beta}} + |u|^{p_s^\ast - 2}\, u \quad \text{in } \R^N,
\end{equation}
where $p < q < p_s^\ast$, $\alpha, \beta > 0$, and $\lambda, \mu \ge 0$ are parameters.

\begin{corollary} \label{Corollary 1.4}
Assume that
\[
\frac{\alpha}{\beta} > \frac{p}{q}
\]
and let $\lambda \ge 0$. Given any $m \ge 1$, $\exists \mu_{m,\,\lambda} > 0$ such that for all $\mu > \mu_{m,\,\lambda}$, equation \eqref{1.6} has $m$ distinct pairs of nontrivial solutions $\pm u^\mu_j,\, j = 1,\dots,m$ such that $E(u^\mu_j) > 0$ and $u^\mu_j \to 0$ in $D^{s,\,p}(\R^N)$ as $\mu \nearrow \infty$. In particular, the number of solutions goes to infinity as $\mu \nearrow \infty$.
\end{corollary}

Theorem \ref{Theorem 1.1}, Theorem \ref{Theorem 1.3}, and Corollary \ref{Corollary 1.4} will be proved in Section \ref{Section 3} after some preliminaries in the next section. Related results can be found in \cite{MR4790797,MR4145649,MR3831057,MR3538875}.

\section{Preliminaries}

\subsection{Weighted Lebesgue spaces}

Let $\gamma > 0$ and let
\[
w(x) = \frac{1}{1 + |x|^{N + \gamma}}, \quad x \in \R^N.
\]
For $q \in (1,\infty)$, let $L^q(\R^N,w)$ denote the weighted Lebesgue space of measurable functions $u : \R^N \to \R$ for which
\begin{equation} \label{2.1}
\pnorm[q,w]{u} := \left(\int_{\R^N} w(x)\, |u|^q\, dx\right)^{1/q} < \infty.
\end{equation}
We have the following embedding result.

\begin{lemma} \label{Lemma 2.1}
$D^{s,\,p}(\R^N)$ is embedded in $L^q(\R^N,w)$ continuously for $q \in [1,p_s^\ast]$ and compactly for $q \in [1,p_s^\ast)$.
\end{lemma}

\begin{proof}
First we note that $w\in L^r(\R^N)$ for $1\le r\le \infty$. Let $u\in D^{s,p}(\R^N)$, then $u\in L^{p_s^*}(\R^N)$, and thus for some positive constant we get for $1\le q< p_s^*$
$$
\int_{\R^N} w|u|^q\,dx\le |w|_{\frac{p_s^*}{p^*-q}}|u|_{p_s^*}^q\le c |w|_{\frac{p_s^*}{p_s^*-q}}[u]_{s,p}^q,
$$
that is,
$$
|u|_{q,w}\le c |w|_{\frac{p_s^*}{p_s^*-q}}^{\frac{1}{q}}[u]_{s,p},
$$
which shows that $i_w: D^{s,p}(\R^N)\to L^q(\R^N, w)$ is linear and continuous for $1\le q< p_s^*$. If $q=p_s^*$, then due to $w\in L^\infty(\R^N)$, the inequality
$$
|u|_{p_s^*,w}\le c[u]_{s,p}
$$
is trivially satisfied, which proves the first part.

As for the compact embedding for $1\le q < p_s^*$, let $(u_n)\subset D^{s,p}(\R^N)$ be bounded. Since $D^{s,p}(\R^N)$ is reflexive, there exists a weakly convergent subsequence still denoted by $(u_n)$, that is, $u_n\rightharpoonup u$ in $D^{s,p}(\R^N)$, which due to $D^{s,p}(\R^N)\hookrightarrow L^{p_s^*}(\R^N)$ yields $u_n\rightharpoonup u$ in $L^{p_s^*}(\R^N)$. On the other hand $(u_n)\subset W^{s,p}(B_R)$ is bounded for any $R$, where $B_R=B(0,R)$ is the open ball with radius $R$, and $W^{s,p}(B_R)$ is the fractional Sobolev space defined by
$$
W^{s,p}(B_R)=\left\{u\in L^p(B_R): \iint_{B_R\times B_R}\frac{|u(x)-u(y)|^p}{|x-y|^{N+sp}}\,dxdy<\infty\right\},
$$
endowed with the norm
$$
|u|_{W^{s,p}(B_R)}=\left(|u|^p_{L^p(B_R)}+[u]_{s,p,B_R}^p\right)^{\frac{1}{p}},
$$
where
$$
[u]_{s,p,B_R}^p=\iint_{B_R\times B_R}\frac{|u(x)-u(y)|^p}{|x-y|^{N+sp}}\,dxdy.
$$
The space $W^{s,p}(B_R)$ is a uniformly convex Banach space and thus reflexive. Also we have $L^{p_s^*}(\R^N)\hookrightarrow L^{p_s^*} (B_R)\hookrightarrow L^{q}(B_R)$, hence $u_n\rightharpoonup u$ in $L^{q}(B_R)$. Since $(u_n)$ is bounded in $W^{s,p}(B_R)$ and $W^{s,p}(B_R)\hookrightarrow\hookrightarrow L^q(B_R)$ is compactly embedded for $1\le q< p_s^*$ (see, e.g., \cite[Theorem 4.54]{MR2895178}), there is a subsequence again denoted by $(u_n)$ such that $u_n\to v$ in $L^q(B_R)$ which by passing again to a subsequence yields $u_n(x)\to v(x)$ for a.a. $x\in B_R$. The weak convergence $u_n\rightharpoonup u$ in $L^{q}(B_R)$ along with $u_n(x)\to v(x)$ for a.a. $x\in B_R$ implies that $u=v$ in $B_R$. Let us show that in fact $u_n\to u$ in $L^q(\R^N, w)$.

For $\varepsilon>0$ arbitrarily be given and any $R>0$, we consider
\begin{equation}\label{S-203}
|u_n-u|_{q,w}^q=\int_{\R^N\setminus B_R} w|u_n-u|^q\,dx+\int_{B_R} w|u_n-u|^q\,dx.
\end{equation}
Since $(u_n)$ is bounded in $D^{s,p}(\R^N)$, that is $[u_n]_{s,p}\le c$, it is also bounded in $L^{p_s^*}(\R^N)$, and thus with some generic constant $c$ independent of $n$ and $R$, we can estimate the first integral on the right-hand side of (\ref{S-203}) as follows:
\begin{eqnarray*}
&&\int_{\R^N\setminus B_R} w|u_n-u|^q\,dx \le c\int_{\R^N\setminus B_R} w\big(|u_n|^q+|u|^q\big)\,dx\\
&&\le c |w|_{L^{\frac{p_s^*}{p_s^*-q}}(\R^N\setminus B_R)}\Big(|u_n|^q_{L^{p_s^*}(\R^N\setminus B_R)}+|u|^q_{L^{p_s^*}(\R^N\setminus B_R)}\Big)\\
&&\le c |w|_{L^{\frac{p_s^*}{p_s^*-q}}(\R^N\setminus B_R)}\Big(|u_n|^q_{p_s^*}+|u|^q_{p_s^*}\Big),
\end{eqnarray*}
which, together with the estimate $|u_n|_{p_s^*}\le c[u_n]_{s,p}\le c$, yields
\begin{equation}\label{S-204}
\int_{\R^N\setminus B_R} w|u_n-u|^q\,dx\le c\,|w|_{L^{\frac{p_s^*}{p_s^*-q}}(\R^N\setminus B_R)}.
\end{equation}
The right-hand side of (\ref{S-204}) can be further estimated as
\begin{equation}\label{S-205}
|w|_{L^{\frac{p_s^*}{p_s^*-q}}(\R^N\setminus B(0,R))}^{\frac{p_s^*}{p_s^*-q}}\le c\int_R^\infty\Big(\frac{1}{1+\varrho^{N+\gamma}}\Big)^{\frac{p_s^*}{p_s^*-q}}\varrho^{N-1}\,d\varrho\le c R^{-(N+\gamma){\frac{p_s^*}{p_s^*-q}}+N},
\end{equation}
since $-(N+\gamma){\frac{p_s^*}{p_s^*-q}}+N<0$. It follows from (\ref{S-204}) and (\ref{S-205}) the existence of $R>0$ sufficiently large such that
\begin{equation}\label{S-206}
\int_{\R^N\setminus B_R} w|u_n-u|^q\,dx< \frac{\varepsilon}{2}, ~~\forall\ n\in \mathbb{N}.
\end{equation}
In view of $u_n\to u$ (strongly) in $ L^q(B_R)$ and taking into account that $w\in L^\infty(\R^N)$ one gets
\begin{equation}\label{S-207}
\int_{B(0,R)} w|u_n-u|^q\,dx< \frac{\varepsilon}{2} \ \mbox{ for $n$ sufficiently large}.
\end{equation}
The estimates (\ref{S-206}) and (\ref{S-207}) complete the proof.
\end{proof}

The following corollary is immediate.

\begin{corollary} \label{Corollary 2.2}
Let $a, b : \R^N \to \R$ be measurable functions satisfying \eqref{1.2}. If $u_j \wto u$ in $D^{s,\,p}(\R^N)$ and $u_j \to u$ a.e., then
\[
\int_{\R^N} a(x)\, |u_j|^p\, dx \to \int_{\R^N} a(x)\, |u|^p\, dx, \qquad \int_{\R^N} b(x)\, |u_j|^q\, dx \to \int_{\R^N} b(x)\, |u|^p\, dx
\]
and
\begin{eqnarray*}
&&\int_{\R^N} a(x)\, |u_j|^{p - 2}\, u_j v\, dx \to \int_{\R^N} a(x)\, |u|^{p - 2}\, uv\, dx,\\
&&\int_{\R^N} b(x)\, |u_j|^{q - 2}\, u_j v\, dx \to \int_{\R^N} b(x)\, |u|^{q - 2}\, uv\, dx
\end{eqnarray*}
for all $v \in D^{s,\,p}(\R^N)$.
\end{corollary}

\subsection{Palais-Smale condition}

We recall that the variational functional $E$ given in \eqref{1.3} satisfies the Palais-Smale compactness condition at the level $c \in \R$, or the \PS{c} condition for short, if every sequence $\seq{u_j} \subset D^{s,\,p}(\R^N)$ satisfying $E(u_j) \to c$ and $E'(u_j) \to 0$, called a \PS{c} sequence for $E$, has a strongly convergent subsequence.

\begin{lemma} \label{Lemma 2.3}
$E$ satisfies the {\em \PS{c}} condition for all $c < \dfrac{s}{N}\, S^{N/sp}$.
\end{lemma}

\begin{proof}
Let $\seq{u_j}$ be a \PS{c} sequence for $E$. Then
\begin{multline} \label{2.2}
E(u_j) = \frac{1}{p} \int_{\R^{2N}} \frac{|u_j(x) - u_j(y)|^p}{|x - y|^{N+sp}}\, dx dy - \frac{\lambda}{p} \int_{\R^N} a(x)\, |u_j|^p\, dx - \frac{\mu}{q} \int_{\R^N} b(x)\, |u_j|^q\, dx\\[5pt]
- \frac{1}{p_s^\ast} \int_{\R^N} |u_j|^{p_s^\ast}\, dx = c + \o(1)
\end{multline}
and
\begin{multline} \label{2.3}
E'(u_j)\, v = \int_{\R^{2N}} \frac{|u_j(x) - u_j(y)|^{p - 2}\, (u_j(x) - u_j(y))\, (v(x) - v(y))}{|x - y|^{N + sp}}\, dx dy\\[5pt]
- \lambda \int_{\R^N} a(x)\, |u_j|^{p - 2}\, u_j v\, dx - \mu \int_{\R^N} b(x)\, |u_j|^{q - 2}\, u_j v\, dx - \int_{\R^N} |u_j|^{p_s^\ast - 2}\, u_j v\, dx = \o(\norm{v})
\end{multline}
for all $v \in D^{s,\,p}(\R^N)$.

First we show that $\seq{u_j}$ is bounded. Taking $v = u_j$ in \eqref{2.3} gives
\begin{multline} \label{2.4}
\int_{\R^{2N}} \frac{|u_j(x) - u_j(y)|^p}{|x - y|^{N+sp}}\, dx dy - \lambda \int_{\R^N} a(x)\, |u_j|^p\, dx - \mu \int_{\R^N} b(x)\, |u_j|^q\, dx\\[5pt]
- \int_{\R^N} |u_j|^{p_s^\ast}\, dx = \o(\norm{u_j}).
\end{multline}
Dividing \eqref{2.4} by $q$ and subtracting from \eqref{2.2} gives
\begin{multline*}
\left(\frac{1}{p} - \frac{1}{q}\right) \left[\int_{\R^{2N}} \frac{|u_j(x) - u_j(y)|^p}{|x - y|^{N+sp}}\, dx dy - \lambda \int_{\R^N} a(x)\, |u_j|^p\, dx\right]\\[5pt]
+ \left(\frac{1}{q} - \frac{1}{p_s^\ast}\right) \int_{\R^N} |u_j|^{p_s^\ast}\, dx = c + \o(1 + \norm{u_j}).
\end{multline*}
This implies that $\norm{u_j}$ is bounded since $1 < p < q < p_s^\ast$,
\[
\int_{\R^N} a(x)\, |u_j|^p\, dx \le \left(\int_{\R^N} a(x)^{N/sp}\, dx\right)^{sp/N} \left(\int_{\R^N} |u_j|^{p_s^\ast}\, dx\right)^{p/p_s^\ast}
\]
by the H\"{o}lder inequality, and $a \in L^{N/sp}(\R^N)$ by \eqref{1.2}.

Since $\seq{u_j}$ is bounded, it converges weakly to some $u \in D^{s,\,p}(\R^N)$ for a renamed subsequence. For a further subsequence, $u_j \to u$ a.e. Set $\widetilde{u}_j = u_j - u$. We will show that if $\norm{\widetilde{u}_j} \ge \eps_0$ for some $\eps_0 > 0$, then $c \ge \dfrac{s}{N}\, S^{N/sp}$, contrary to assumption.

In view of Corollary \ref{Corollary 2.2}, \eqref{2.2} and \eqref{2.4} reduce to
\begin{multline} \label{2.5}
\frac{1}{p} \int_{\R^{2N}} \frac{|u_j(x) - u_j(y)|^p}{|x - y|^{N+sp}}\, dx dy - \frac{\lambda}{p} \int_{\R^N} a(x)\, |u|^p\, dx - \frac{\mu}{q} \int_{\R^N} b(x)\, |u|^q\, dx\\[5pt]
- \frac{1}{p_s^\ast} \int_{\R^N} |u_j|^{p_s^\ast}\, dx = c + \o(1)
\end{multline}
and
\begin{multline} \label{2.6}
\int_{\R^{2N}} \frac{|u_j(x) - u_j(y)|^p}{|x - y|^{N+sp}}\, dx dy - \lambda \int_{\R^N} a(x)\, |u|^p\, dx - \mu \int_{\R^N} b(x)\, |u|^q\, dx\\[5pt]
- \int_{\R^N} |u_j|^{p_s^\ast}\, dx = \o(1),
\end{multline}
respectively. Moreover, taking $v = u$ in \eqref{2.3} and passing to the limit using Corollary \ref{Corollary 2.2} gives
\begin{equation} \label{2.7}
\int_{\R^{2N}} \frac{|u(x) - u(y)|^p}{|x - y|^{N+sp}}\, dx dy - \lambda \int_{\R^N} a(x)\, |u|^p\, dx - \mu \int_{\R^N} b(x)\, |u|^q\, dx - \int_{\R^N} |u|^{p_s^\ast}\, dx = 0.
\end{equation}
We have
\[
\int_{\R^N} \left(\frac{|u_j(x) - u_j(y)|^p}{|x - y|^{N+sp}} - \frac{|u(x) - u(y)|^p}{|x - y|^{N+sp}}\right) dx dy = \int_{\R^{2N}} \frac{|\widetilde{u}_j(x) - \widetilde{u}_j(y)|^p}{|x - y|^{N+sp}}\, dx dy + \o(1)
\]
as in Perera et al.\! \cite[Lemma 3.2]{MR3458311} and
\[
\int_{\R^N} \left(|u_j|^{p_s^\ast} - |u|^{p_s^\ast}\right) dx = \int_{\R^N} |\widetilde{u}_j|^{p_s^\ast}\, dx + \o(1)
\]
by the Br\'{e}zis-Lieb lemma \cite[Theorem 1]{MR699419}, so subtracting \eqref{2.7} from \eqref{2.6} gives
\[
\int_{\R^{2N}} \frac{|\widetilde{u}_j(x) - \widetilde{u}_j(y)|^p}{|x - y|^{N+sp}}\, dx dy = \int_{\R^N} |\widetilde{u}_j|^{p_s^\ast}\, dx + \o(1).
\]
Combining this with $\eqref{1.5}$ and using $\norm{\widetilde{u}_j} \ge \eps_0$, we get
\begin{equation} \label{2.8}
\int_{\R^{2N}} \frac{|\widetilde{u}_j(x) - \widetilde{u}_j(y)|^p}{|x - y|^{N+sp}}\, dx dy \ge S^{N/sp} + \o(1).
\end{equation}
On the other hand, dividing \eqref{2.6} by $p_s^\ast$ and subtracting from \eqref{2.5} gives
\[
\frac{s}{N} \left[\int_{\R^{2N}} \frac{|u_j(x) - u_j(y)|^p}{|x - y|^{N+sp}}\, dx dy - \lambda \int_{\R^N} a(x)\, |u|^p\, dx\right] - \mu \left(\frac{1}{q} - \frac{1}{p_s^\ast}\right) \int_{\R^N} b(x)\, |u|^q\, dx = c + \o(1),
\]
and multiplying \eqref{2.7} by $s/N$ and subtracting from this gives
\[
\frac{s}{N} \int_{\R^{2N}} \frac{|\widetilde{u}_j(x) - \widetilde{u}_j(y)|^p}{|x - y|^{N+sp}}\, dx dy + \mu \left(\frac{1}{p} - \frac{1}{q}\right) \int_{\R^N} b(x)\, |u|^q\, dx + \frac{s}{N} \int_{\R^N} |u|^{p_s^\ast}\, dx = c + \o(1).
\]
Since $q > p$ and $\mu \ge 0$, this together with \eqref{2.8} implies that $c \ge \dfrac{s}{N}\, S^{N/sp}$.
\end{proof}

\subsection{Cohomological index}

We recall that the $\Z_2$-cohomological index of Fadell and Rabinowitz \cite{MR0478189} is defined as follows. Let $W$ be a Banach space and let $\A$ denote the class of symmetric subsets of $W \setminus \set{0}$. For $A \in \A$, let $\overline{A} = A/\Z_2$ be the quotient space of $A$ with each $u$ and $-u$ identified, let $f : \overline{A} \to \RP^\infty$ be the classifying map of $\overline{A}$, and let $f^\ast : H^\ast(\RP^\infty) \to H^\ast(\overline{A})$ be the induced homomorphism of the Alexander-Spanier cohomology rings. The cohomological index of $A$ is defined by
\[
i(A) = \begin{cases}
\sup \set{m \ge 1 : f^\ast(\omega^{m-1}) \ne 0} & \text{if } A \ne \emptyset\\[5pt]
0 & \text{if } A = \emptyset,
\end{cases}
\]
where $\omega \in H^1(\RP^\infty)$ is the generator of the polynomial ring $H^\ast(\RP^\infty) = \Z_2[\omega]$.

\begin{example}
The classifying map of the unit sphere $S^{m-1}$ in $\R^m,\, m \ge 1$ is the inclusion $\RP^{m-1} \incl \RP^\infty$, which induces isomorphisms on $H^q$ for $q \le m - 1$, so $i(S^{m-1}) = m$.
\end{example}

The cohomological index has the following so called piercing property, which is not shared by the Krasnoselskii’s genus.

\begin{proposition}[see {\cite[Proposition (3.9)]{MR0478189}}]
If $A, A_0, A_1 \in \A$ are closed and $\varphi : A \times [0,1] \to A_0 \cup A_1$ is a continuous map such that $\varphi(-u,t) = - \varphi(u,t)$ for all $(u,t) \in A \times [0,1]$, $\varphi(A \times [0,1])$ is closed, $\varphi(A \times \set{0}) \subset A_0$, and $\varphi(A \times \set{1}) \subset A_1$, then
\[
i(\varphi(A \times [0,1]) \cap A_0 \cap A_1) \ge i(A).
\]
\end{proposition}

\subsection{An abstract critical point theorem}

Let $W$ be a Banach space and let $E \in C^1(W,\R)$ be an even functional, i.e., $E(-u) = E(u)$ for all $u \in W$. Assume that there exists $c^\ast > 0$ such that $E$ satisfies the \PS{c} condition for all $c \in (0,c^\ast)$. Let $\A^\ast$ denote the class of symmetric subsets of $W$ and let $\Gamma$ denote the group of odd homeomorphisms of $W$ that are the identity outside the set $\set{u \in W : 0 < E(u) < c^\ast}$. For $\rho > 0$, the pseudo-index of $M \in \A^\ast$ related to $i$, $S_\rho = \set{u \in W : \norm{u} = \rho}$, and $\Gamma$ is defined by
\[
i^\ast(M) = \min_{\gamma \in \Gamma}\, i(\gamma(M) \cap S_\rho)
\]
(see Benci \cite{MR84c:58014}). Making essential use of the piercing property of the cohomological index, the following critical point theorem was proved in Yang and Perera \cite{MR3616328} (see also Perera and Szulkin \cite{MR2153141} and Perera et al.\! \cite[Proposition 3.44]{MR2640827}).

\begin{theorem}[see {\cite[Theorem 2.4]{MR3616328}}] \label{Theorem 2.6}
Let $A_0$ and $B_0$ be symmetric subsets of the unit sphere $S = \set{u \in W : \norm{u} = 1}$ such that $A_0$ is compact, $B_0$ is closed, and
\[
i(A_0) \ge k + m - 1, \qquad i(S \setminus B_0) \le k - 1
\]
for some $k, m \ge 1$. Assume that there exists $R > \rho$ such that
\begin{equation} \label{2.9}
\sup_{u \in A}\, E(u) \le 0 < \inf_{u \in B}\, E(u), \qquad \sup_{u \in X}\, E(u) < c^\ast,
\end{equation}
where $A = \set{Ru : u \in A_0}$, $B = \set{\rho u : u \in B_0}$, and $X = \set{tu : u \in A_0,\, 0 \le t \le R}$. For $j = k,\dots,k + m - 1$, let $\A_j^\ast = \set{M \in \A^\ast : M \text{ is compact and } i^\ast(M) \ge j}$ and set
\[
c_j^\ast := \inf_{M \in \A_j^\ast}\, \max_{u \in M}\, E(u).
\]
Then
\[
\inf_{u \in B}\, E(u) \le c_k^\ast \le \dotsb \le c_{k+m-1}^\ast \le \sup_{u \in X}\, E(u),
\]
each $c_j^\ast$ is a critical value of $E$, and there are $m$ distinct pairs of associated critical points.
\end{theorem}

We will use this theorem to prove our multiplicity results.

\subsection{Eigenvalue problem}

We recall the construction of variational eigenvalues based on the cohomological index for the eigenvalue problem \eqref{1.4}, which can be written as
\[
\Ap[u] = \lambda \Bp[u]
\]
in the dual $D^{s,\,p}(\R^N)^\ast$ of $D^{s,\,p}(\R^N)$, where $\Ap, \Bp \in C(D^{s,\,p}(\R^N),D^{s,\,p}(\R^N)^\ast)$ are the operators given by
\begin{eqnarray*}
&&\dualp{\Ap[u]}{v} = \int_{\R^{2N}} \frac{|u(x) - u(y)|^{p - 2}\, (u(x) - u(y))\, (v(x) - v(y))}{|x - y|^{N + sp}}\, dx dy,\\
&&\dualp{\Bp[u]}{v} = \int_{\R^N} a(x)\, |u|^{p - 2}\, uv\, dx, \quad u, v \in D^{s,\,p}(\R^N).
\end{eqnarray*}
Since $D^{s,\,p}(\R^N)$ is uniformly convex and
\[
\dualp{\Ap[u]}{v} \le \norm{u}^{p-1} \norm{v}, \quad \dualp{\Ap[u]}{u} = \norm{u}^p \quad \forall u, v \in D^{s,\,p}(\R^N)
\]
by the H\"{o}lder inequality, it follows from Perera et al.\! \cite[Proposition 1.3]{MR2640827} that $\Ap$ is of type (S), i.e., every sequence $\seq{u_j} \subset D^{s,\,p}(\R^N)$ such that $u_j \wto u$ and $\dualp{\Ap[u_j]}{u_j - u} \to 0$ has a subsequence that converges strongly to $u$. The operator $\Bp$ is compact by Lemma \ref{Lemma 2.1}.

Let
\[
I(u) = \int_{\R^{2N}} \frac{|u(x) - u(y)|^p}{|x - y|^{N+sp}}\, dx dy, \quad J(u) = \int_{\R^N} a(x)\, |u|^p\, dx, \quad u \in D^{s,\,p}(\R^N),
\]
let $S = \set{u \in D^{s,\,p}(\R^N) : \norm{u} = 1}$ be the unit sphere in $D^{s,\,p}(\R^N)$,
and let
\[
S^+ = \set{u \in S : J(u) > 0}.
\]
Then $S^+$ is a symmetric open submanifold of $S$ and eigenvalues of problem \eqref{1.4} coincide with critical values of the $C^1$-functional
\[
\Psi(u) = \frac{1}{J(u)}, \quad u \in S^+
\]
(see \cite[Chapter 9]{MR2640827}). Let $\F$ denote the class of symmetric subsets of $S^+$ and let $i(M)$ be the cohomological index of $M \in \F$. The following theorem was proved in \cite{MR2640827}.

\begin{theorem}[see {\cite[Theorem 9.2]{MR2640827}}] \label{Theorem 2.7}
For $k \ge 1$, let $\F_k = \set{M \in \F : i(M) \ge k}$ and set
\[
\lambda_k := \inf_{M \in \F_k}\, \sup_{u \in M}\, \Psi(u).
\]
Then $\lambda_k \nearrow \infty$ is a sequence of eigenvalues of problem \eqref{1.4}.
\begin{enumroman}
\item \label{Theorem 2.7.i} The first eigenvalue is given by
    \[
    \lambda_1 = \min_{u \in S^+}\, \Psi(u) > 0.
    \]
\item \label{Theorem 2.7.ii} If $\lambda_k = \dotsb = \lambda_{k+m-1} = \lambda$ and $E_\lambda$ is the set of eigenfunctions associated with $\lambda$ that lie on $S^+$, then
    \[
    i(E_\lambda) \ge m.
    \]
\item \label{Theorem 2.7.iii} If $\lambda_{k-1} < \lambda < \lambda_k$, then
    \[
    i(\Psi^{\lambda_{k-1}}) = i(S^+ \setminus \Psi_\lambda) = i(\Psi^\lambda) = i(S^+ \setminus \Psi_{\lambda_k}) = k - 1,
    \]
    where $\Psi^a = \set{u \in S^+ : \Psi(u) \le a}$ and $\Psi_a = \set{u \in S^+ : \Psi(u) \ge a}$ for $a \in \R$.
\end{enumroman}
\end{theorem}

\subsection{A regularity result}

We have the following regularity result.

\begin{theorem} \label{Theorem 2.8}
Let $u\in D^{s,p}(\R^N)$ be a solution of the fractional $p$-Laplacian equation
\begin{equation}\label{Eq-1}
(-\Delta_p)^s u=f(x,u),
\end{equation}
where $f$ satisfies the condition:
\begin{equation}\label{Eq-2}
|f(x,t)|\le b_2(x)|t|^{p-1}+b_3(x),
\end{equation}
with coefficient functions $b_2$ and $b_3$ that are supposed to satisfy
the following hypotheses:
\begin{itemize}
\item[(H1)] $b_2(x)=m(x)+n(x);\, m\in L^{t_0}(\R^N)$, for some $t_0>\frac{N}{ps}$, and \, $n\in L^{\frac{N}{ps}}(\R^N)$,
\item [(H2)] $b_3\in L^{\frac{\beta}{p-1}}(\R^N)\cap L^\infty(\R^N)$, for some $\beta > p$.
\end{itemize}
Suppose $u\in L^\beta(\R^N)$ and $r\in (p^{*},\infty)$. There exists $\epsilon_0=\epsilon_0 (r,p,s,N)$ such that
$$
\|n\|_{\frac{N}{ps}}\leq \epsilon_0\Longrightarrow\|u\|_{r}\le C \max\left\{\|u\|_{\beta}, \|u\|_{\beta}^{\theta_0}\right\},
$$
for some $\theta_0=\theta_0(p,\beta,N,s)$ with $0<\theta_0\le 1$ and $C=C\left(p,N,s, \|m\|_{t_0}, \|b_3\|\right)$, with $\|b_3\|:=\|b_3\|_{\frac{\beta}{p-1}}+\|b_3\|_{\infty}$.
 
In particular if $n(x)=0$, then $u\in L^{p^{*}}(\R^N)\cap L^\infty(\R^N)$ and the estimate 
\begin{equation}\label{Eq-3}
\|u\|_{r}\le C \max\left\{\|u\|_{\beta}, \|u\|_{\beta}^{\theta_0}\right\},
\end{equation}
is valid for all $p^{*}\leq r\leq \infty$.
\end{theorem}
\begin{proof}For $a\in\R$ and $\gamma>0,\, L>1$ we introduce the functions
$$
a^{1+\gamma}:=a|a|^{\gamma}\,\quad\quad a_L^{1+\gamma}=a\min\{|a|,L\}^\gamma.
$$
Note that this definition satisfies the following simple properties, which will be freely used throughout the presentation below.\begin{equation}\label{Eq-3.5}
|a_L^{1+\gamma}|=|a|_L^{1+\gamma},\,\quad\quad |a_L^{1+\gamma}|\leq |a|^{1+\gamma},\, \quad\quad |a|_L^{(1+\gamma_1)(1+\gamma_2)}\leq (|a|_L^{1+\gamma_1})^{1+\gamma_2}.
\end{equation}
Next we recall the following inequality (see Iannizzotto et al.\! \cite[Lemma 2.3]{MR4012292});
$$
\left|a^{1+\gamma}-b^{1+\gamma}\right|^p\le C(1+p\gamma)^{p-1}(a-b)^{p-1}\left(a^{1+p\gamma}-b^{1+p\gamma}\right),
$$
By considering different possibilities as to the location of $a$ and $b$ relative to $L$ on the real line, one can easily verify the following related inequality 
\begin{equation}\label{Eq-4}
\left|a_L^{1+\gamma}-b_L^{1+\gamma}\right|^p\le C(1+p\gamma)^{p-1}(a-b)^{p-1}\left(a_L^{1+p\gamma}-b_L^{1+p\gamma}\right),
\end{equation}
from which it follows 
\begin{equation}\label{Eq-5}
\left|u_L^{1+\gamma}(x)-u_L^{1+\gamma}(y)\right|^p\le C(L,p,\gamma)|u(x)-u(y)|^p.
\end{equation}
Substituting $a=u(x)$ and $b=u(y)$ in inequalities (\ref{Eq-4}) and (\ref{Eq-5}), integrating in $\R^N$ using the measure $d\mu=\frac{1}{|x-y|^{N+sp}}dxdy$ we conclude that $u_L^{1+\gamma}\in D^{s,p}(\R^N)$ and 
\begin{equation}\label{Eq-6}
\left[u_L^{1+\gamma}\right]^p_{s,p}\le C(p)(1+p\gamma)^{p-1}\left\langle (-\Delta_p)^su,u_L^{1+p\gamma}\right\rangle
\end{equation}
Next let $\alpha=\beta-(p-1)>1$, and $L>1$. Applying inequality (\ref{Eq-6}) with $\gamma=\frac{\alpha-1}{p}$, using (\ref{Eq-2}), (\ref{Eq-3.5}) and the fractional Sobolev-Gagliardo inequality, and noting that $1+\gamma =\frac{\beta}{p}$, we have
\begin{eqnarray}
\|u_L^{\frac{\beta}{p}}\|_{p^{*}}^p &\leq & C(s,p)\beta^{p-1} \int_{\R^N}|f(x,u)||u|^p\min\{|u|,L\}^{\alpha -1}\,dx\nonumber\\
& \leq & C(s,p)\beta^{p-1}\left[\int_{\R^N}m(x)(|u_L|^{\frac{\beta}{p}})^p+ \int_{\R^N}n(x))(|u_L|^{\frac{\beta}{p}})^p+ \int_{\R^N}b_3(x)|u|^{\alpha}\right]\label{Eq-7}
\end{eqnarray}
We now estimate the integrals on the right-hand side of (\ref{Eq-7}). We recall that by hypotheses (H1), $t_0>\frac{N}{ps}$ and thus $p< \frac{pt_0}{t_0-1}<p^*=\frac{Np}{N-sp}$. Hence, by interpolation along with Young's inequality we obtain the following estimate for the first integral on the right-hand side of (\ref{Eq-7}):
\begin{eqnarray}
I_1&=&\int_{\R^N}m(x)(|u_L|^{\frac{\beta}{p}})^p\,dx \le \|m\|_{t_0}\left( \int_{\R^N} (|u_L|^{\frac{\beta}{p}})^{\frac{pt_0}{t_0-1}} \right) ^{\frac{t_0-1}{t_0}}\nonumber \\
&\le& \|m\|_{t_0}\left(\varepsilon \|u_L^{\frac{\beta}{p}}\|_{p^{*}}^{p} +
C(\varepsilon) \|u_L^{\frac{\beta}{p}}\|_{p}^{p}\right),\label{Eq-8}
\end{eqnarray}
where $\varepsilon$ is arbitrary and $C(\varepsilon)=\epsilon^{\frac{-N}{pst_0-N}}$. Next using Holder inequality we estimate the next two integrals on the right-hand side of (\ref{Eq-7}) as follows
\begin{equation}\label{Eq-9}
 I_2=\int_{\R^N}n(x) (|u_L|^{\frac{\beta}{p}})^p\,dx\leq \|n\|_{\frac{N}{ps}}\|u_L^{\frac{\beta}{p}}\|_{p^{*}}^{p}
\end{equation}
\begin{equation}\label{Eq-10}
 I_3=\int_{\R^N}b_3(x) |u|^{\alpha}\,dx\leq \|b_3\|_{\frac{\beta}{p-1}}\left(\int_{\R^N}|u|^\beta dx\right)^{\frac{\alpha}{\beta}}
\end{equation}
Now taking $\varepsilon=\frac{1}{4C(s,p)\beta^{p-1}\|m\|_{t_0}}$ in (\ref{Eq-8}) and imposing the condition
\begin{equation}\label{Eq-11}
\|n\|_{\frac{N}{ps}}\leq \frac{1}{4C(s,p)\beta^{p-1}}
\end{equation}
on the coefficient function $n(x)$, in view of (\ref{Eq-7})-(\ref{Eq-11}) we obtain
\begin{equation*}
\|u_L^{\frac{\beta}{p}}\|_{p^{*}}^{p}\leq \tilde{C}(p,s)\beta^{(p-1)\frac{pt_0s}{pt_0s-N}}\|m\|_{t_0}\|u\|_{\beta}^{\beta}+C(s,p)\beta^{p-1} \|b_3\|\|u\|_\beta^\alpha.
\end{equation*}
Since by (\ref{Eq-3.5}) 
$$
\left(|u_L(x)|^{\frac{\beta}{p}}\right)^{p^{*}}\geq |u_L(x)|^{\frac{\beta}{p}p^{*}}=|u_L(x)|^{\frac{\beta N}{N-ps}}
$$
we conclude
\begin{equation*}
\|u_L\|_{\frac{\beta N}{N-sp}}\leq (C \beta)^{\frac{\sigma}{\beta}}\max\left\{\|u\|_\beta, \|u\|_\beta^{1-\frac{p-1}{\beta}}\right\},
\end{equation*}
Letting $L\to\infty$, we finally arrive at the following inequality:
\begin{equation}\label{Eq-12}
\|u\|_{\frac{\beta N}{N-sp}}\leq (C \beta)^{\frac{\sigma}{\beta}}\max\left\{\|u\|_\beta, \|u\|_\beta^{1-\frac{p-1}{\beta}}\right\},
\end{equation}
where $C=C(p, s,N, \|m\|_{t_0},\|b_3\|)$ and $\sigma=\sigma(p,s,N,t_0)>1$.

It is now clear that one may iterate inequality (\ref{Eq-12}) $k$ times by taking $\beta_0=\beta,\, \, \beta_k=\chi^k\beta\, $, with $\chi=\frac{N}{N-sp}$, if
$$
\|n\|_{\frac{N}{ps}}\leq \frac{1}{4C(s,p)(\chi^{k-1}\beta)^{p-1}}
$$
and obtain
$$ |u|_{\chi^k\beta}\leq (C\chi\beta)^{\tilde{\sigma}}\max\left\{\|u\|_{\beta},|u|_{\beta}^{\theta(k)}\right\},
$$
where $\tilde{\sigma}= \frac{\sigma}{\beta}(1+\sum_{m=1}^{\infty} \frac{m}{\chi^m})$ and
$$\theta(k)=\prod_{i=0}^{k-1} R(i),\quad\quad R(i)=1\, \, \mbox{ or }\, \, R(i)=1-\frac{p-1}{\chi^i \beta}, 
$$
Note that
$$0< \theta_0=\theta_0(p,\beta, N) :=\prod_{i=0}^{\infty}(1-\frac{p-1}{\chi^i \beta})\leq \theta(k)\leq 1
$$
since $\sum_{i=0}^{\infty} \frac{p-1}{\chi^i \beta}<\infty$. Hence given $r\in (p^{*},\infty)$, we let $k\in \N$ be such that
$\chi^{k-1}\beta\leq r<\chi{k}\beta $, then the above argument implies that 
$$
\|n\|_{\frac{N}{ps}}\leq \epsilon_0:=\frac{1}{4C(s,p)(\chi^{k-1}\beta)^{p-1}}\Longrightarrow\|u\|_{r}\le C \max\left\{\|u\|_{\beta}, \|u\|_{\beta}^{\theta_0}\right\},
$$
with $C$ and $\theta_0$ as in the statement of the theorem. Finally the last statement of the theorem easily follows as we will have 
$$\|u\|_{r}\le C \max\left\{\|u\|_{\beta}, \|u\|_{\beta}^{\theta_0}\right\},
$$ 
for all $p^{*}\leq r$, with $C$ and $\theta_0$ independent of $r$, and therefore for $r=\infty$ as well. The proof of the theorem is now complete.
\end{proof}

\section{Proofs of Theorem \ref{Theorem 1.1}, Theorem \ref{Theorem 1.3}, and Corollary \ref{Corollary 1.4}} \label{Section 3}

We prove Theorem \ref{Theorem 1.1} and Theorem \ref{Theorem 1.3} by applying Theorem \ref{Theorem 2.6} to the functional
\[
E(u) = \frac{1}{p}\, \big(I(u) - \lambda\, J(u)\big) - \frac{\mu}{q} \int_{\R^N} b(x)\, |u|^q\, dx - \frac{1}{p_s^\ast} \int_{\R^N} |u|^{p_s^\ast}\, dx, \quad u \in D^{s,\,p}(\R^N),
\]
taking
\[
c^\ast = \frac{s}{N}\, S^{N/sp}
\]
(see Lemma \ref{Lemma 2.3}).

\subsection{Proof of Theorem \ref{Theorem 1.1}}

Let $\lambda \in (\lambda_k - \delta,\lambda_k)$, let
\begin{equation} \label{3.1}
0 < \eps < \min \set{\delta/(\lambda_k - \lambda) - 1,(\lambda_{k+m} - \lambda_{k+m-1})/\delta},
\end{equation}
and let $\eps_\lambda = \eps\, (\lambda_k - \lambda)$. Then $\lambda_{k+m-1} < \lambda_{k+m-1} + \eps_\lambda < \lambda_{k+m}$ and hence
\[
i(S^+ \setminus \Psi_{\lambda_{k+m-1} + \eps_\lambda}) = k + m - 1
\]
by Theorem \ref{Theorem 2.7} \ref{Theorem 2.7.iii}. Since $S^+ \setminus \Psi_{\lambda_{k+m-1} + \eps_\lambda}$ is an open symmetric subset of $S$, it then has a compact symmetric subset $A_0$ of index $k + m - 1$ (see the proof of Proposition 3.1 in Degiovanni and Lancelotti \cite{MR2371112}). We take $B_0 = \Psi_{\lambda_k} \cup (S \setminus S^+)$. Then
\[
S \setminus B_0 = S^+ \setminus \Psi_{\lambda_k}.
\]
Either $\lambda_1 = \cdots = \lambda_k$, or $\lambda_{l-1} < \lambda_l = \cdots = \lambda_k$ for some $2 \le l \le k$. In the former case,
\[
i(S \setminus B_0) = i(S^+ \setminus \Psi_{\lambda_1}) = i(\emptyset) = 0 \le k - 1
\]
by Theorem \ref{Theorem 2.7} \ref{Theorem 2.7.i}. In the latter case,
\[
i(S \setminus B_0) = i(S^+ \setminus \Psi_{\lambda_l}) = l - 1 \le k - 1
\]
by Theorem \ref{Theorem 2.7} \ref{Theorem 2.7.iii}.

Let $R > \rho > 0$ and let
\[
A = \set{Ru : u \in A_0}, \quad B = \set{\rho u : u \in B_0}, \quad X = \set{tu : u \in A_0,\, 0 \le t \le R}.
\]
For $t \ge 0$,
\begin{equation} \label{3.2}
E(tu) = \begin{cases}
\ds{\frac{t^p}{p} \left(1 - \frac{\lambda}{\Psi(u)}\right) - \frac{\mu\, t^q}{q} \int_{\R^N} b(x)\, |u|^q\, dx - \frac{t^{p_s^\ast}}{p_s^\ast} \int_{\R^N} |u|^{p_s^\ast}\, dx}, & u \in S^+\\[25pt]
\ds{\frac{t^p}{p} - \frac{\mu\, t^q}{q} \int_{\R^N} b(x)\, |u|^q\, dx - \frac{t^{p_s^\ast}}{p_s^\ast} \int_{\R^N} |u|^{p_s^\ast}\, dx}, & u \in S \setminus S^+.
\end{cases}
\end{equation}
Since $\int_{\R^N} |u|^{p_s^\ast}\, dx > 0$ on $S$, this integral is bounded away from zero on the compact set $A_0$, so \eqref{3.2} gives
\[
E(tu) \le \frac{t^p}{p} - c_1\, t^{p_s^\ast} \quad \forall u \in A_0
\]
for some constant $c_1 > 0$. So the first inequality in \eqref{2.9} holds if $R$ is sufficiently large. Since $S$ is bounded, \eqref{3.2} together with Lemma \ref{Lemma 2.1} gives
\[
E(tu) \ge \begin{cases}
\ds{\frac{t^p}{p} \left(1 - \frac{\lambda}{\lambda_k}\right) - c_2\, t^q - c_3\, t^{p_s^\ast}} & \forall u \in \Psi_{\lambda_k}\\[20pt]
\ds{\frac{t^p}{p} - c_2\, t^q - c_3\, t^{p_s^\ast}} & \forall u \in S \setminus S^+
\end{cases}
\]
for some constants $c_2, c_3 > 0$, so the second inequality in \eqref{2.9} holds if $\lambda < \lambda_k$ and $\rho$ is sufficiently small.

Any $u \in X$ can be written as $u = t \widetilde{u}$ for some $\widetilde{u} \in A_0$ and $t \in [0,R]$. Then
\[
I(u) = t^p\, I(\widetilde{u}) = t^p, \qquad J(u) = t^p\, J(\widetilde{u}) = \frac{I(u)}{\Psi(\widetilde{u})}.
\]
Since $A_0 \subset S^+ \setminus \Psi_{\lambda_{k+m-1} + \eps_\lambda}$ and hence $\Psi(\widetilde{u}) < \lambda_{k+m-1} + \eps_\lambda = \lambda_k + \eps_\lambda$, this gives
\[
I(u) \le (\lambda_k + \eps_\lambda)\, J(u),
\]
so
\[
E(u) \le \frac{1}{p}\, (\lambda_k + \eps_\lambda - \lambda)\, J(u) - \frac{1}{p_s^\ast} \int_{\R^N} |u|^{p_s^\ast}\, dx.
\]
Since $\eps_\lambda = \eps\, (\lambda_k - \lambda)$ and
\[
J(u) \le \pnorm[N/sp]{a} \pnorm[p_s^\ast]{u}^p
\]
by the H\"{o}lder inequality, this in turn gives
\[
E(u) \le \frac{1}{p}\, (1 + \eps)(\lambda_k - \lambda) \pnorm[N/sp]{a} \tau^p - \frac{1}{p_s^\ast}\, \tau^{p_s^\ast},
\]
where $\tau = \pnorm[p_s^\ast]{u}$. Maximizing the right-hand side over all $\tau \ge 0$ gives
\[
\sup_{u \in X}\, E(u) \le \frac{s}{N} \left[(1 + \eps)(\lambda_k - \lambda) \pnorm[N/sp]{a}\right]^{N/sp} < \frac{s}{N} \left[\delta \pnorm[N/sp]{a}\right]^{N/sp} = \frac{s}{N}\, S^{N/sp}
\]
by \eqref{3.1}. It now follows from Theorem \ref{Theorem 2.6} that the equation \eqref{1.1} has $m$ distinct pairs of solutions $\pm u^\lambda_j,\, j = 1,\dots,m$ such that
\begin{equation} \label{3.3}
0 < E(u^\lambda_j) \le \frac{s}{N} \left[(1 + \eps)(\lambda_k - \lambda) \pnorm[N/sp]{a}\right]^{N/sp}.
\end{equation}

Next we show that $u^\lambda_j \to 0$ in $D^{s,\,p}(\R^N)$ as $\lambda \nearrow \lambda_k$. By \eqref{3.3},
\begin{multline} \label{3.4}
E(u^\lambda_j) = \frac{1}{p} \int_{\R^{2N}} \frac{|u^\lambda_j(x) - u^\lambda_j(y)|^p}{|x - y|^{N+sp}}\, dx dy - \frac{\lambda}{p} \int_{\R^N} a(x)\, |u^\lambda_j|^p\, dx - \frac{\mu}{q} \int_{\R^N} b(x)\, |u^\lambda_j|^q\, dx\\[5pt]
- \frac{1}{p_s^\ast} \int_{\R^N} |u^\lambda_j|^{p_s^\ast}\, dx = \o(1)
\end{multline}
as $\lambda \nearrow \lambda_k$. Since $u^\lambda_j$ is a critical point of $E$,
\begin{multline} \label{3.5}
E'(u^\lambda_j)\, u^\lambda_j = \int_{\R^{2N}} \frac{|u^\lambda_j(x) - u^\lambda_j(y)|^p}{|x - y|^{N+sp}}\, dx dy - \lambda \int_{\R^N} a(x)\, |u^\lambda_j|^p\, dx - \mu \int_{\R^N} b(x)\, |u^\lambda_j|^q\, dx\\[5pt]
- \int_{\R^N} |u^\lambda_j|^{p_s^\ast}\, dx = 0.
\end{multline}
Dividing \eqref{3.5} by $p$ and subtracting from \eqref{3.4} gives
\[
\mu \left(\frac{1}{p} - \frac{1}{q}\right) \int_{\R^N} b(x)\, |u^\lambda_j|^q\, dx + \frac{s}{N} \int_{\R^N} |u^\lambda_j|^{p_s^\ast}\, dx = \o(1).
\]
Since
\[
\int_{\R^N} a(x)\, |u^\lambda_j|^p\, dx \le \pnorm[N/sp]{a} \left(\int_{\R^N} |u^\lambda_j|^{p_s^\ast}\, dx\right)^{p/p_s^\ast}
\]
by the H\"{o}lder inequality, the desired conclusion now follows from \eqref{3.4}.

Finally we use our regularity result (see Theorem \ref{Theorem 2.8}) to show that $u^\lambda_j \to 0$ in $L^\infty(\R^N)$ as $\lambda \nearrow \lambda_k$. Note that $w=u_{\lambda}$ solves the following fractional $p$-Laplacian equation
\begin{equation}\label{Eq-13}
(-\Delta_p)^s w=f_{\lambda}(x,w),
\end{equation}
where 
$$f_{\lambda}(x,s)=\left( \lambda\, a(x) + \mu\, b(x)\, |u_{\lambda}(x)|^{q - p}+ |u_{\lambda}(x)|^{p^{*} - p}\right) |s|^{p-2}s.
$$ 
In addition,
$$|f_{\lambda}(x,s)|\leq b_{\lambda, 2}|s|^{p-1}=(m_{\lambda}(x)+n_{\lambda}(x))|s|^{p-1},$$
where
$$m_{\lambda}(x):=\lambda_k\, a(x) + \mu\, b(x)\, |u_{\lambda}(x)|^{q - p}\, \quad\mbox{ and }\quad n_{\lambda}(x):=|u_{\lambda}(x)|^{p^{*} - p}.
$$
In order to apply Theorem \ref{Theorem 2.8}, we need to verify hypothesis (H1). As $p<q<p^{*}$, we have 
$\frac{N}{sp}(q-p)<p^{*}=\frac{Np}{N-sp}$, from which follows the existence of $t_0>\frac{N}{sp}$ satisfying $t_0(q-p)<p^{*}$. Hence the fact that $\|u_{\lambda} \|_{p^{*}}\rightarrow 0$ as $\lambda \rightarrow \lambda_k$ allows for the estimate
$$ \|m_{\lambda}\|_{t_0}\leq \lambda_k\|a\| +\mu \|b\|\|u_{\lambda}\|_{p^{*}}^{q-p}\leq C $$
for some $C$ independent of $\lambda$, with $\|a\|=\|a\| _{1}+\|a\| _{\infty}$ and $\|b\|=\|b\| _{1}+\|b\| _{\infty}$.
On the other hand,
$$
\|n_{\lambda}\|_{\frac{N}{ps}}=\|u_{\lambda}\|_{p^{*}}^{\frac{p^2s}{N-ps}}.
$$ 
So (H1) is valid and since $\|n_{\lambda}\|_{\frac{N}{ps}}\rightarrow 0$ as $\lambda \rightarrow \lambda_k$, an application of Theorem \ref{Theorem 2.8} yields
$$
\|u_{\lambda}\|_{r}\leq C \max\left\{\|u_{\lambda}\|_{\beta}, \|u_{\lambda}\|_{\beta}^{\theta_0}\right\}\rightarrow 0,\, \quad \mbox{ as }\, \quad \lambda \rightarrow \lambda_k
$$
for any $r\in [p^{*},\infty)$. But this implies that, apostriori, $n_{\lambda}=|u_{\lambda}^{p^{*}-p}|\in L^{t}(\R^N)$ for all $t\in [\frac{N}{ps},\infty)$, and, moreover, $\|u_{\lambda}\|_{p^{*}+\delta}\rightarrow 0$ for any $\delta>0$. Therefore the growth on the right-hand side of equation (\ref{Eq-13}) can be recast as
$$
|f_{\lambda}(x,s)|=\left( \lambda\, a(x) + \mu\, b(x)\, |u_{\lambda}(x)|^{q - p}+ |u_{\lambda}(x)|^{p^{*} - p}\right) |s|^{p-1}\leq \tilde{m}_{\lambda}(x)|s|^{p-1},
$$ 
where 
$$
\tilde{m}_{\lambda}(x):=\lambda_k\, a(x) + \mu\, b(x)\, |u_{\lambda}(x)|^{q - p}+ |u_{\lambda}(x)|^{p^{*} - p}.
$$
Note that, with $t_0>\frac{N}{ps}$ chosen above, writing $(p^{*}-p)t_0=p^{*}+\delta$, we have
$$\|\tilde{m}_{\lambda}\|_{t_0}\leq \lambda_k\|a\| +\mu \|b\|\|u_{\lambda}\|_{p^{*}}^{q-p}+\|u_{\lambda}\|_{p^{*}+\delta}^{\frac{P^{*}+\delta}{p^{*}-p}}\leq C
$$
for some $C$ independent of $\lambda$. Finally, a further application of Theorem \ref{Theorem 2.8} in the case of $n(x)=0$ yields
$$
\|u_{\lambda}\|_{\infty}\leq C \max\left\{\|u_{\lambda}\|_{p^{*}}, \|u_{\lambda}\|_{p^{*}}^{\theta_0}\right\}\rightarrow 0,\, \quad \mbox{ as }\, \quad \lambda \rightarrow \lambda_k,
$$
completing the proof.

\subsection{Proof of Theorem \ref{Theorem 1.3}}

Set $\lambda_0 = 0$. Then $\lambda_{k-1} \le \lambda < \lambda_k$ for some $k \ge 1$. We take $B_0 = \Psi_{\lambda_k} \cup (S \setminus S^+)$. Then
\[
i(S \setminus B_0) = i(S^+ \setminus \Psi_{\lambda_k}) = k - 1
\]
by Theorem \ref{Theorem 2.7} \ref{Theorem 2.7.i} \& \ref{Theorem 2.7.iii}. By taking $m$ larger if necessary, we may assume that $\lambda_{k+m-1} < \lambda_{k+m}$. Then
\[
i(S^+ \setminus \Psi_{\lambda_{k+m}}) = k + m - 1
\]
by Theorem \ref{Theorem 2.7} \ref{Theorem 2.7.iii}. Since $S^+ \setminus \Psi_{\lambda_{k+m}}$ is an open symmetric subset of $S$, it then has a compact symmetric subset $A_0$ of index $k + m - 1$ (see the proof of Proposition 3.1 in Degiovanni and Lancelotti \cite{MR2371112}).

Let $R > \rho > 0$ and let
\[
A = \set{Ru : u \in A_0}, \quad B = \set{\rho u : u \in B_0}, \quad X = \set{tu : u \in A_0,\, 0 \le t \le R}.
\]
As in the proof of Theorem \ref{Theorem 1.1}, the first inequality in \eqref{2.9} holds if $R$ is sufficiently large and the second inequality holds if $\rho$ is sufficiently small. Any $u \in X$ can be written as $u = t \widetilde{u}$ for some $\widetilde{u} \in A_0$ and $t \in [0,R]$. Then
\[
I(u) = t^p\, I(\widetilde{u}) = t^p, \qquad J(u) = t^p\, J(\widetilde{u}) = \frac{I(u)}{\Psi(\widetilde{u})}.
\]
Since $A_0 \subset S^+ \setminus \Psi_{\lambda_{k+m}}$ and hence $\Psi(\widetilde{u}) < \lambda_{k+m}$, this gives
\[
I(u) \le \lambda_{k+m}\, J(u),
\]
so
\[
E(u) \le \frac{1}{p}\, (\lambda_{k+m} - \lambda)\, J(u) - \frac{\mu}{q} \int_{\R^N} b(x)\, |u|^q\, dx.
\]
Since
\[
J(u) \le \pnorm[q/(q-p)]{\frac{a}{b^{p/q}}} \pnorm[q,b]{u}^p
\]
by the H\"{o}lder inequality (see \eqref{2.1}), this in turn gives
\[
E(u) \le \frac{1}{p}\, (\lambda_{k+m} - \lambda) \pnorm[q/(q-p)]{\frac{a}{b^{p/q}}} \tau^p - \frac{\mu}{q}\, \tau^q,
\]
where $\tau = \pnorm[q,b]{u}$. Maximizing the right-hand side over all $\tau \ge 0$ gives
\[
\sup_{u \in X}\, E(u) \le \left(\frac{1}{p} - \frac{1}{q}\right) \left[(\lambda_{k+m} - \lambda) \pnorm[q/(q-p)]{\frac{a}{b^{p/q}}}\right]^{q/(q-p)} \mu^{-p/(q-p)}.
\]
If $\mu > 0$ is sufficiently large, the last expression is less than $\dfrac{s}{N}\, S^{N/sp}$ and hence it follows from Theorem \ref{Theorem 2.6} that the equation \eqref{1.1} has $m$ distinct pairs of solutions $\pm u^\mu_j,\, j = 1,\dots,m$ such that
\[
0 < E(u^\mu_j) \le \left(\frac{1}{p} - \frac{1}{q}\right) \left[(\lambda_{k+m} - \lambda) \pnorm[q/(q-p)]{\frac{a}{b^{p/q}}}\right]^{q/(q-p)} \mu^{-p/(q-p)}.
\]
An argument similar to that in the proof of Theorem \ref{Theorem 1.1} shows that $u^\mu_j \to 0$ in $D^{s,\,p}(\R^N)$ as $\mu \nearrow \infty$.

\subsection{Proof of Corollary \ref{Corollary 1.4}}

Assume $\frac{\alpha}{\beta}>\frac{p}{q}$, then by using spherical coordinates and taking into account that $q>p$ we get the estimate
\begin{eqnarray*}
\int_{\R^N}\left(\frac{a}{b^{p/q}}\right)^{\frac{q}{q-p}}dx
&=&\int_{\R^N}\left(\frac{\left(1+|x|^{N+\beta}\right)^{\frac{p}{q}}}{1+|x|^{N+\alpha}}\right)^{\frac{q}{q-p}}dx\\
&=&\int_{B(0,1)} \frac{ \left(1+|x|^{N+\beta}\right)^{\frac{p}{q-p}}}{ \left(1+|x|^{N+\alpha}\right)^{\frac{q}{q-p}}}dx
+\int_{\R^N\setminus B(0,1)} \frac{ \left(1+|x|^{N+\beta}\right)^{\frac{p}{q-p}}}{ \left(1+|x|^{N+\alpha}\right)^{\frac{q}{q-p}}}dx\\
&\le &1+\int_{\R^N\setminus B(0,1)} \frac{ \left(1+|x|^{N+\beta}\right)^{\frac{p}{q-p}}}{ \left(1+|x|^{N+\alpha}\right)^{\frac{q}{q-p}}}dx\\
&\le &1+\int_1^{\infty}\varrho^{(N+\beta) \frac{p}{q-p}-( N+\alpha)\frac{q}{q-p}+N-1 }d\varrho<\infty,
\end{eqnarray*}
since 
$$
(N+\beta) \frac{p}{q-p}-( N+\alpha)\frac{q}{q-p}+N=\frac{1}{q-p}(\beta p-\alpha q)<0,
$$
and thus $\frac{a}{b^{p/q}} \in L^{q/(q-p)}(\R^N)$.

\hspace{1pt}

{\bf Data availability statement:} This manuscript does not use any data.

\hspace{1pt}

{\bf Conflict of interest statement:} The authors declare that there are no conflicts of interest.

\def\cprime{$''$}

\end{document}